\documentclass{amsart}
\usepackage[utf8]{inputenc}
\usepackage[all]{xy}
\usepackage{amssymb,amsmath,amsfonts,amscd,enumerate,verbatim,calc,amsthm}
\usepackage{latexsym}
\usepackage{amsthm,amsfonts,amssymb,mathrsfs}
\usepackage{rotating}
\usepackage[leqno]{amsmath}
\usepackage{xspace}
\usepackage[all]{xy}
\usepackage{longtable}
\usepackage[colorlinks=true]{hyperref}
\usepackage[all]{xy}
\input xy
\xyoption{all}

\headsep=1cm \numberwithin{equation}{section}
\newtheorem{thm}{Theorem}[section]

\newtheorem{rem}[thm]{Remark}
\newtheorem{ques}[thm]{Question}

\newtheorem{Con*}[thm]{Conjectuer}

\newcommand{\Ann}{\mbox{Ann}\,}
\newcommand{\coker}{\mbox{coker}\,}
\newcommand{\Max}{\mbox{Max}\,}

\newcommand{\Att}{\mbox{Att}\,}

\newcommand{\E}{\mbox{E}}

\renewcommand{\H}{\mbox{H}}
\newcommand{\V}{\mbox{V}}

\newcommand{\fb}{\mathfrak{b}}

\newcommand{\fm}{\mathfrak{m}}
\newcommand{\fp}{\mathfrak{p}}

\newcommand\m{\operatorname{\frak m}}

\newcommand\Rad{\operatorname{Rad}}

\def\Hom{\operatorname{\mathsf{Hom}}}
\def\dim{\operatorname{\mathsf{dim}}}

\DeclareMathOperator{\Supp}{Supp}

\DeclareMathOperator{\Ass}{Ass}

\begin{document}

\title{On the dimension of cofinite modules}

\author[M. Rahro Zargar and Ghadeh Ghasemi]{Majid Rahro Zargar and Ghader Ghasemi}

\address{Majid Rahro Zargar, Department of Engineering Sciences, Faculty of Advanced Technologies, University of Mohaghegh Ardabili, Namin, Ardabil, Iran,}
\email{zargar9077@gmail.com}
\email{m.zargar@uma.ac.ir}
\address{Ghader Ghasemi, Faculty of Mathematical Sciences, Department of Mathematics, University of Mohaghegh Ardabili, Ardabil, Iran,}
\email{ghaderghasemi54@gmail.com}

\begin{abstract}
Let $I$ be an ideal of a commutative Noetherian complete local ring $R$. In the present paper, we establish the equality $\dim R/(I+\Ann_R M)=\dim M$ for all $I$-cofinite $R$-modules $M$.
\end{abstract}
\maketitle


\section{Introduction}

Throughout this paper, let $R$ denote a commutative Noetherian ring
(with identity) and $I$ be an ideal of $R$. Also, we denote by $\mathscr{C}(R)$ the category of all $R$-modules. For an $R$-module $M$, the
$i$th local cohomology module of $M$ with respect to $I$ is
defined as:
$$H^i_I(M) = \underset{n\geq1} {\varinjlim}\,\,
\text{Ext}^i_R(R/I^n, M).$$ We refer the reader to \cite{Gr1} or
\cite{BS} for more details about local cohomology. Hartshorne, in \cite{Ha}, defined an $R$-module $M$ to be
$I$-{\it cofinite}, if $\Supp M\subseteq
V(I)$ and ${\rm Ext}^{i}_{R}(R/I, M)$ is a finitely generated $R$-module
for all $i\geq0$.

In the sequel, for any ideal $I$ of $R$, we denote by $\mathscr{C}(R, I)_{cof}$ the category of all $I$-cofinite
$R$-modules. Also, we denote by $\mathscr{A}(R)$ the class of all ideals $I$ of $R$ such that $\mathscr{C}(R, I)_{cof}$ is an Abelian subcategory of $\mathscr{C}(R)$; that is, if $f: M\longrightarrow N$ is an $R$-homomorphism of
$I$-cofinite modules, then the $R$-modules ${\rm ker}\, f$ and ${\rm coker}\, f$ are $I$-cofinite, too.

One of the more elementary results concerning the finitely generated $R$-module $M$ is the relation $\dim R/\Ann_R M=\dim M$. This relation easily follows from the fact that $\Supp M=V(\Ann_R M)$. But, this result does not hold for cofinite $R$-modules with respect to an ideal of $R$. For example, if $(R,\m,k)$ is a Noetherian local ring of dimension $d>0$, then the $\m$-cofinite $R$-module $\E_R(k)$ (the injective envelope of the residue field $k=R/\m$ of $R$) is faithful, that is $\Ann_R \E_R(k)=0$, and $\dim \E_R(k)=0$. So, $\dim R/\Ann_R \E_R(k)=d>\dim \E_R(k)$.\\
But, we know that every finitely generated $R$-module $M$ is $(0_R)$-cofinite, where $(0_R)$ denotes the zero ideal of $R$. So, one has the equality $\dim R/((0_R)+\Ann_R M)=\dim M$. Furthermore, in \cite{Me1}, Melkersson proved that if $(R,\m)$ is a Noetherian complete local ring and $I$ is an ideal of $R$, then  each zero-dimensional $I$-cofinite module $M$ is Artinian and for each $\fp\in \Att_R M$ we have $\dim R/(I+\fp)=0$. So, $\dim R/(I+\Ann_R M)=0=\dim M$. Therefore, according to the above, it is natural to raise the following question:
\begin{ques} Let $I$ be an ideal of a Notherian complete local ring $R$ and $M$ an $I$-cofinite $R$-module. Then, does the equality $\dim R/(I+\Ann_R M)=\dim M$ hold?
\end{ques}

More recently, in \cite{GP1}, Pirmohammadi provided a partial answer to the above question. Indeed. he proved that for any ideal $I$ of a Noetherian complete local ring with $I\in\mathscr{A}(R)$ and any non-zero $I$-cofinite $R$-module $M$, the equality $\dim R/(I+\Ann_R M)=\dim M$ holds.

In the present paper, we provide an affirmative answer to the above question without any additional assumption. Indeed, we prove the following result:
\begin{thm}
Let $I$ be and ideal of a complete local ring $R$ and $M$ a non-zero $I$-cofinite $R$-module. Then \emph{$$\dim R/(I+\Ann_R M)=\dim M.$$}
\end{thm}
Here, we should notice that the category of $I$-cofinite $R$-modules with dimension less than one is Abelian (see \cite[Theorem 2.7]{BNS} ). But, the category of $I$-cofinite $R$-modules with a dimension greater than two is not necessarily Abelian. For example, consider the following counterexample of Hartshorne: Let $R=k[[x,y,z,w]]$, $I=(x,z)$ and $M=R/(xy-zw)$. Now, set $t:=xy-zw$ and by the exact sequence $0\longrightarrow R \stackrel{t}\longrightarrow R\longrightarrow M \longrightarrow0$, one can get the following induced exact sequence: $$\cdots\longrightarrow \H_{I}^2(R) \stackrel{t^*}\longrightarrow \H_{I}^2(R)\longrightarrow \H_{I}^2(M) \longrightarrow0.$$

Since $\H_{I}^i(R)=0$ for all $i\neq 2$, one can use \cite[Proposition 2.1]{MZ} to see that the $R$-module $\H_{I}^2(R)$ is $I$-cofinite. However, $\coker(t^*)=\H_{I}^2(M)$ is not $I$-cofinite (see \cite{Ha}). Furthermore, using the fact that $\Supp_R(\H_{I}^2(R))=\Supp_R(\underset{i\in\Bbb{N}_{0}}\bigoplus\H_{I}^i(R))=\Supp_R(R/I)$, implies that $\dim_R(\H_{I}^2(R))=2$.

\section{results}

The starting point of this section is the following main result.
\begin{thm}
\label{2.5}
Let $I$ be and ideal of a complete local ring $(R,\m)$ and $M$ a non-zero $I$-cofinite $R$-module. Then \emph{$$\dim R/(I+\Ann_R M)=\dim M.$$}
\end{thm}
\begin{proof} We use induction on $\dim_R M$ to prove the result. First set $\fb:=\Ann_R M$, and then assume that $\dim M=0$. Then, using \cite[Lemma 2.1]{Me2} implies that $M$ is Artinian; and thus $\dim R/(I+\fb)= 0=\dim M$, by \cite[Lemma 2.1]{B3}. Now, let $\dim M \geq 1$ and suppose that the result has been proved for all cofinite $R$-modules of dimension $\dim_R M-1.$ Note that since $\Supp M\subseteq V(I+\fb)$, then $\dim R/(I+\fb)\geq \dim M.$ Therefore, we may and do assume that $n:=\dim R/(I+\fb)\geq 0.$ Hence, there are elements $x_1,\dots, x_n$ in $R-(I+\fb)$ such that $$\Rad((x_1 +(I+\fb),\dots, x_n +(I+\fb)))=\fm/(I+\fb),$$ and so $\Rad((x_1,\dots,x_n) +(I+\fb))=\fm$.
By \cite[Lemma 2.4]{B3}, $\Gamma_{\m}(M)$ is an Artinian $I$-cofinite $R$-module. Therefore, by a similar argument in the first step of induction, we have $\Rad(I+\Ann_R \Gamma_{\m}(M))\supseteq\m$. Now, set $\overline{M}:=M/\Gamma_{\m}(M)$ and consider the following short exact sequence:
$$0\longrightarrow \Gamma_{\m}(M) \longrightarrow M \longrightarrow \overline{M}\longrightarrow 0,$$
which shows that the $R$-module $\overline{M}$ is $I$-cofinite and $\dim \overline{M}=\dim_R M$. Furthermore, since $\Rad(\fb)=\Rad(\Ann_R \Gamma_{\fm}(M)\cap\Ann_R\overline{M})$ and $\Rad(I+\Ann_R \Gamma_{\m}(M))\supseteq\m$, then by considering the following equalities:

\[\begin{array}{rl}
\Rad(I+\fb)&=\Rad(I+\Rad(\fb))\\
&=\Rad(I+\Rad(\Ann_R\Gamma_{\fm}(M)\cap\Ann_R\overline{M}))\\
&=\Rad(I+\Ann_R\Gamma_{\fm}(M)\cap\Ann_R\overline{M})\\
&=\Rad(I+\Ann_R\Gamma_{\fm}(M))\cap\Rad(I+\Ann_R\overline{M})\\
&=\fm\cap\Rad(I+\Ann_R\overline{M})\\
&=\Rad(I+\Ann_R\overline{M}),
\end{array}\]
one has $\dim R/(I+\fb)=\dim R/(I+\Ann_R \overline{M})$. Therefore, by replacing $M$ with $\overline{M}$, we may and do assume that $\Gamma_{\fm}(M)=0$, and so $\fm\not\in\Ass_R M$. Therefore, one can deduce that $$(x_1,\dots,x_n)\nsubseteq \bigcup_{\fp\in\Ass_R M}\fp.$$
So, there exists an element $b\in (x_2,\dots,x_n)$ such that $z:=x_1+b\notin\bigcup_{\fp\in\Ass_R M}\fp,$ and also $(x_1,x_2,\dots,x_n)=(z,x_2,\dots,x_n)$. Hence, by the following equalities:
\[\begin{array}{rl}
\fm/(I+\fb)&=\Rad(\frac{(x_1,\dots,x_n) +(I+\fb)}{(I+\fb)})\\
&=\Rad(\frac{(z, x_2,\dots,x_n)+(I+\fb)}{I+\fb})\\
&=\Rad((z+(I+\fb), x_2 +(I+\fb),\dots, x_n +(I+\fb))),
\end{array}\]
one can deduce that the element $z+(I+\fb)$ is a subset of a system of parameters for $R/(I+\fb)$.
Here, we should notice that $z\not\in(I+\fb)$. It therefore follows from \cite[Proposition 15.22]{SHA} that
\begin{alignat}{2}
\dim R/(I+\fb)=\dim R/(Rz+I+\fb)+1.\tag{2.1}
\end{alignat}
Next, in view of \cite[Lemma 4.2]{B2} we have that the $R$-module $H^1_{Rz}(M)$ is a non-zero $(I+Rz)$-cofinite module and $\dim_R \H^1_{Rz}(M)=\dim_R M-1$. Therefore, by the inductive hypothesis one has the following equality:
\begin{equation}
\dim R/(I+Rz+\Ann_R \H^1_{Rz}(M))=\dim_R \H^1_{Rz}(M)=\dim_R M-1.\tag{2.2}
\end{equation}
Now, we claim that $z\not\in\bigcup_{\fp\in\Ass_R M/(0:_MI^n)}\fp$ for each $n\in\Bbb{N}$. To do this, contrary assume that $z\in\bigcup_{\fp\in\Ass_R M/(0:_MI^n)}\fp$ for some $t\in\Bbb{N}$. Hence, there is $\fp\in\Ass_R M/(0:_MI^t)$ such that $z\in\fp$ and also there is an element $\beta \in M$ with $\beta\not\in (0:_MI^t)$ such that $\fp=(0:_R \beta+(0:_MI^t))$. Therefore $z\beta I^t=0$; and thus $I^t\beta\subseteq (0:_M Rz)=0$. So, $\beta\in (0:_MI^t)$ which is a contradiction. Now, for each $n\in\Bbb{N}$, by the exact sequence $$0\longrightarrow(0:_MI^n)\longrightarrow M\longrightarrow M/(0:_MI^n)\longrightarrow0,$$ one get the exact sequence $0\longrightarrow H^1_{Rz}((0:_MI^n))\longrightarrow H^1_{Rz}(M),$ which shows that $\Ann_R H^1_{Rz}(M)\subseteq \Ann_R H^1_{Rz}((0:_MI^n)).$ Since $z$ is a regular element on the finitely generated $R$-module $(0:_MI^n)$, then by \cite[Theorem 3.4]{B1} one has the equality $\Ann_R H^1_{Rz}((0:_MI^n))=\Ann_R(0:_MI^n).$
Hence, by considering the following
\[\begin{array}{rl}
\Ann_R M &\subseteq \Ann_R H^1_{Rz}(M)\\
&\subseteq\bigcap_{n=1}^{\infty} \Ann_R(0:_MI^n)\\
&=\Ann_R M,
\end{array}\]
we can get the equality $\Ann_R H^1_{Rz}(M)\overset{\dag}=\Ann_R M$. Therefore, given the (2.1), (2.2) and $^\dag$ we have the following equalities:
\[\begin{array}{rl}
\dim R/(I+\fb)&=\dim R/(Rz+I+\fb)+1\\
&=\dim R/(Rz+I+\Ann_R H^1_{Rz}(M))+1\\
&=\dim_R M -1+1\\
&=\dim_R M,
\end{array}\]
as required.
\end{proof}
\begin{rem}\emph{There exists a two-dimensional Noetherian local domain $(R,\fm)$
which does not have a maximal Cohen-Macaulay $R$-module (see \cite[§1]{HOK} and \cite{FR}).
Therefore, by \cite[Corollay 3.6]{B4}, over such a ring the $R$-module $\H_{\fm}^1(R)$ can not be finitely generated.
Now, let $x$ be a nonzero element of $R$ and consider the following exact sequence:
$$0\longrightarrow \Gamma_{\fm}(R/xR)\longrightarrow \H_{\fm}^1(R)\stackrel{x}\longrightarrow \H_{\fm}^1(R),$$
which is induced from the exact sequence $0\longrightarrow R\stackrel{x}\longrightarrow R\longrightarrow R/xR\longrightarrow 0.$
Therefore, one has the following isomorphism:
\begin{alignat}{2}
\Gamma_{\fm}(R/xR)\cong\Hom_R(R/Rx, \H^1_{\fm}(R)).\tag{2.3}
\end{alignat}
If $x\in\Ann_R(\H_{\fm}^1(R))$, then by the above isomorphism $\Gamma_{\fm}(R/xR)\cong \H_{\fm}^1(R)$
which is a contradiction by the fact that $\H_{\fm}^1(R)$ is not finitely generated, and so we have $\Ann_R(\H_{\fm}^1(R))=0$.\\
Furthermore,  by the isomorphism (2.3), $\Hom_R(R/Rx, \H^1_{\fm}(R))$ is finitely generated and also $\Supp_R\H^1_{\fm}(R)\subseteq \Max R\cap\V(Rx).$ Hence, by \cite[Lemma 2.1]{Me1} one has $\H^1_{\fm}(R)$ is $Rx$-cofinite. Therefore, there is the following equality: $$1=\dim_R(R/Rx)=\dim_R R/(Rx+\Ann_R(\H_{\fm}^1(R)).$$ But, $\dim_R\H_{\fm}^1(R)=0$, and so $\dim_R R/(Rx+\Ann_R(\H_{\fm}^1(R))\neq \dim_R\H_{\fm}^1(R)$. It therefore follows from Theorem \ref{2.5} that the desired ring $R$ cannot be complete. So, in our main result, the completeness assumption on $R$ is quite necessary.
}

\end{rem}

$\mathbf{Acknowledgments}$. The authors would like to thank Prof. Kamal Bahmanpour for his valuable and profound comments during the preparation of the manuscript.

\end{document}